\documentclass[12pt]{amsart} 
 \usepackage{amssymb, latexsym, amsmath} 

{\normalsize }

\begin{document} 

 \theoremstyle{plain} 
 \newtheorem{theorem}{Theorem}[section] 
 \newtheorem{lemma}[theorem]{Lemma} 
 \newtheorem{corollary}[theorem]{Corollary} 
 \newtheorem{proposition}[theorem]{Proposition} 

\theoremstyle{definition} 
\newtheorem*{definition}{Definition}
\newtheorem{example}[theorem]{Example}
\newtheorem{remark}[theorem]{Remark}

\title[Skew braces and the Galois correspondence ]{Skew braces and the Galois correspondence for Hopf Galois structures}

\author{Lindsay N. Childs}
\address{Department of Mathematics and Statistics\\
University at Albany\\
Albany, NY 12222}
\email{lchilds@albany.edu}

\date{\today}

\newcommand{\QQ}{\mathbb{Q}} 
\newcommand{\FF}{\mathbb{F}} 
\newcommand{\ZZ}{\mathbb{Z}}
\newcommand{\ZZm}{\mathbb{Z}/m\mathbb{Z}}
\newcommand{\ZZp}{\mathbb{Z}/p\mathbb{Z}}
\newcommand{\ee}{\end{eqnarray}} 
\newcommand{\ben}{\begin{eqnarray*}} 
\newcommand{\een}{\end{eqnarray*}} 
\newcommand{\dis}{\displaystyle} 
\newcommand{\beal}{\[ \begin{aligned}} 
\newcommand{\eeal}{ \end{aligned} \]} 
\newcommand{\lb}{\lambda} 
\newcommand{\Gm}{\Gamma} 
\newcommand{\gm}{\gamma} 
\newcommand{\bpm}{\begin{pmatrix}} 
\newcommand{\epm}{\end{pmatrix}} 
\newcommand{\Fp}{\mathbb{F}_p} 
\newcommand{\Fpx}{\mathbb{F}_p^{\times}} 
\newcommand{\Ann}{\text{Ann} }
\newcommand{\tb}{\textbullet \ }
\newcommand{\GL}{\mathrm{GL}}
\newcommand{\gb}{\genfrac{[}{]}{0pt}{}}
\newcommand{\M}{\mathrm{M}}
\newcommand{\Aut}{\mathrm{Aut}}
\newcommand{\End}{\mathrm{End}}
\newcommand{\Perm}{\mathrm{Perm}}
\newcommand{\PGL}{\mathrm{PGL}}
\newcommand{\diag}{\mathrm{diag}}
\newcommand{\Reg}{\mathrm{Reg}}
\newcommand{\Hol}{\mathrm{Hol}}
\newcommand{\Inn}{\mathrm{Inn}}
\newcommand{\InHol}{\mathrm{InHol}}
\newcommand{\Lc}{\mathcal{L}}
\newcommand{\Hom}{\mathrm{Hom}}

 \begin{abstract} Let $L/K$ be a Galois extension of fields with Galois group $\Gamma$, and suppose $L/K$ is also an $H$-Hopf Galois extension. Using the recently uncovered connection between Hopf Galois structures and skew left braces, we introduce a method to quantify the failure of  surjectivity  of the Galois correspondence from subHopf algebras of $H$ to intermediate subfields of $L/K$,  given by the Fundamental Theorem of Hopf Galois Theory.  Suppose $L \otimes_K H = LN$ where $N \cong (G, \star)$.   Then there exists a skew left brace $(G, \star, \circ)$ where $(G, \circ) \cong \Gamma$.  We show that there is a bijective correspondence between the set of intermediate fields $E$ between $K$ and $L$ that correspond to $K$-subHopf algebras of $H$ and  a set of sub-skew left braces of $G$ that we call the $\circ$-stable subgroups of $(G, \star)$. Counting these subgroups and comparing that number with the number of subgroups of $\Gamma \cong (G, \circ)$ describes how far the Galois correspondence for the $H$-Hopf Galois structure is from being surjective.  The method is illustrated by a variety of examples.
\end{abstract}
\maketitle
\section{Introduction}

Chase and Sweedler [CS69] introduced the concept of a Hopf Galois extension of commutative rings as a generalization of a classical Galois extension of fields $L/K$ with Galois group $\Gamma$.  The idea is to view  the Galois structure on $L$ as an action by the group ring $K\Gamma$, a $K$-Hopf algebra, and then replace $K\Gamma$ by a general cocommutative $K$-Hopf algebra $H$.    In that setting, the Fundamental Theorem of Galois Theory (FTGT) of Chase and Sweedler [CS69] states that if $L/K$ is an $H$-Hopf Galois extension of fields for $H$ a  $K$-Hopf algebra, then there is an injection  from the set of $K$-sub-Hopf algebras of $H$ to the  set of intermediate fields $K \subseteq E \subseteq L$, given by sending a $K$-subHopf algebra $J$ to the fixed ring 
\[ L^J = \{x \in L | h(x)= \epsilon(h)x \text{ for all } h \text{ in } J \}\]
where $\epsilon:  H \to K$ is the counit map.   The \emph{strong form} of the FTGT holds if the injection is also a surjection.  For a classical Galois extension of fields with Galois group $\Gamma$, the FTGT holds in its strong form.   But it is known from [GP87] that in general, the Galois correspondence need not be surjective.  So for any given example it is of interest to determine how far surjectivity fails. 

Let  $L/K$ be a Galois extension of fields with Galois group $\Gamma$, and suppose $L/K$ has an $H$-Hopf Galois structure of type $G$.  Then $H$ corresponds to a regular subgroup of $\Perm(\Gamma)$ isomorphic to $G$ and normalized by $\lambda(\Gamma)$, the image of the left regular representation of $\Gamma$ in $\Perm(\Gamma)$.  

In turn, from [Ch89] and [By96], the Hopf Galois structure corresponds to a regular subgroup of $\Hol(G)$ isomorphic to $\Gamma$, where $\Hol(G) \subset \Perm(G)$ is the holomorph of $G$, that is, the normalizer in $\Perm(G)$ of $\lambda(G)$.   

In turn, by results of increasing generality, from [CDVS06] to [FCC12] to [Bac16], [GV17] and [SV18], a regular subgroup of $\Hol(G)$ isomorphic to $\Gamma$ corresponds to a skew left brace $(G, \star, \circ)$ where $(G, \star) \cong G$ and $(G, \circ) \cong \Gamma$.   

In [Ch17] the question of the image of the Galois correspondence was studied for a Hopf Galois structure of type $G \cong \Gamma$ on a Galois extension $L/K$ with $p$-elementary abelian  Galois group $\Gamma$. The method used the [CDVS06] correspondence between regular subgroups of $\Hol(G)$ and commutative radical algebra structures on $G$.  In that setting the intermediate fields in the image of the Galois correspondence correspond to left ideals of  the radical algebra $A$ with circle group isomorphic to $\Gamma$ and additive group isomorphic to $G$.  Subsequently,  [CG18] obtained upper and lower bounds on the number of left ideals in that setting, hence  upper and lower bounds on the proportion of all intermediate fields that are in the image of the Galois correspondence.
 
In this paper we generalize [Ch17] to the general setting of a skew left brace.  Our main result obtains a bijective correspondence between intermediate fields in the image of the Galois correspondence and what we call $\circ$-stable subgroups of the additive group $(G, \star)$ of the skew left brace $(G,\star, \circ)$ corresponding to the Hopf Galois structure on $L/K$.  

If the skew left brace $G$ is a left brace (that is, the additive group $(G, \star)$ is abelian), then a $\circ$-stable subgroup of $G$ is a left ideal of $G$.

We apply the method first to two examples where the skew left brace is a non-commutative radical algebra, and to the example of the order 8 left brace of [Rum07].  We conclude by applying the method to Hopf Galois structures that correspond to skew left braces arising from an exact factorization of $G$ into $HJ$,  described in [SV18], where the Galois group $\Gamma = H \times J$:  these are examples of Hopf Galois structures arising from fixed point free pairs of homomorphisms from $\Gamma$ to $G$ as in [BC12] and [By15].  

My thanks to Nigel Byott for introducing me to the connections between Hopf Galois theory and brace theory,  to Griff Elder and the University of Nebraska at Omaha for their excellent support of research in this area over the past seven years, and to the University at Albany Three Voices Grant Program for its support.  Thanks also to the referee for a careful reading of the manuscript.

\section{On skew left braces}

\begin{definition}  A finite group $(G, \star)$ is a skew left brace with ``additive group'' $(G, \star)$ if $G$ has an additional group structure $(G, \circ)$  so that  for all $g, h, k$ in $G$, 
\[ g \circ (h\star k) = (g \circ h) \star g^{-1} \star (g \circ k). \]  \end{definition}
Here $g^{-1}$ is the inverse of $g$ in $(G, \star)$.  Let $\overline{g}$ be the inverse of $g$ in $(G, \circ)$ .  

Given a skew left brace $(G, \star, \circ)$, the identities of the groups $(G, \star)$ and $(G, \circ)$ coincide ([GV17, Lemma 1.7]).

Skew left braces were defined by Guarneri and Vendramin in [GV17] and have been studied in  [Bac16a] and [SV18].  For $(G, \star)$ abelian, a skew left brace is a left brace.  Left braces were defined by Rump [Rum07] and have subsequently been studied by numerous authors, largely motivated by the connection with solutions of the Yang-Baxter equation.   If $A$ is a radical algebra, that is, an  associative ring (without unit) $(A, +, \cdot)$ with the property that with the operation $a \circ b = a + b + a \cdot b$, $(A, \circ)$ is a group, then $(A, +, \circ)$ is a left brace.

 Associated to a set $G$ with two group operations $(G, \star, \circ)$  are the two left regular representation maps:
 \beal \lb_\star:  &G \to \Perm(G), \lb_\star(g)(h) = g \star h,\\
 \lb_{\circ}: &G \to \Perm(G), \lb_{\circ}(g)(h) = g \circ h. \eeal
 
 Then $(G, \star, \circ)$ is a skew left brace if and only if for all $g$ in $G$, the map 
\[\Lc_g = (\lb_\star)(g^{-1})\lb_{\circ}(g):  G \to \Perm(G)\]
has image in $\Aut(G, \star)\subset \Perm(G)$.  That is, 
for all $g, h, k$ in $G$, 
\[ \Lc_g(h \star k) = \Lc_g(h) \star\Lc_g(k),\]
or, 
\[ g^{-1}\star (g \circ (h \star k)) =  g^{-1}\star (g \circ h)\star  g^{-1}\star (g \circ  k),\]
equivalent to the defining relation of a left skew brace.

The holomorph $\Hol(G, \star)$ is the normalizer in $\Perm(G)$ of $\lb_\star(G)$:  there is an isomorphism
\[ \iota:  \lb_\star(G) \rtimes \Aut(G, \star) \to \Hol(G, \star) \subset \Perm(G)\]
by 
\[ \iota(\lb_{\star}(g), \theta)(h) = \lb_\star(g)\theta(h) = g \star \theta(h)\]
for $\theta$ in $\Aut(G, \star)$, $g, h$ in $G$.  If $\theta = \Lc_g$, then 
\beal \iota(\lambda_{\star}(g), \Lc_g)(h) &= \lambda_{\star}(g)\Lc_g(h)\\
&= \lambda_{\circ}(g)(h) \\&= g \circ h.\eeal

For $(G, \star, \circ)$ a skew left brace, the map  
\[ \Lc:  (G, \circ) \to \Aut(G, \star)\]
defined by $g\mapsto \Lc_g$ where 
\[   \Lc_g(x) = g^{-1}\star (g \circ x)\]
is a group homomorphism:  see [GV17, Proposition 1.9, Corollary 1.10].   

In the brace literature the map $\Lc_g$ is denoted by $\lambda_g$.   In this paper we reserve $\lambda$ for left regular representation maps.

\section{Hopf Galois extensions}
Suppose $L$ is a Galois extension of $K$, fields, with finite Galois group $\Gamma$.  (Denote the field extension as $L/K$.) This is equivalent to either of two equivalent conditions:  i)   there is an isomorphism of $L$-algebras
\[ h: L \otimes_K L  \to \Hom_L (K\Gamma, L)\]
given by $h(x \otimes y)(\gamma) = x \gamma(y)$  
(a condition equivalent to the condition that $L$ is a splitting field for $L/K$), and ii)  The map 
\[ j: L \otimes_K K\Gamma\to \End_K(L) \]
given by  $j(x \otimes h)(y) = xh(y)$, is a $K$-module isomorphism.  

Suppose $H$ is a cocommutative $K$-Hopf algebra and $L$ is an $H$-module.  Then $L/K$ is an $H$-Hopf Galois extension if  two properties hold.  One is that $L$ is an $H$-module algebra:    for all $h$ in $H$, $x, y$ in $L$, 
\[ h(xy) = \sum_{(h)} h_{(1)}(x)h_{(2)}(y),\]
where the notation for the comultiplication $\Delta:  H \to H \otimes_K H$ is the Sweedler notation
\[ \Delta(h) = \sum_{(h)} h_{(1)} \otimes h_{(2)}.\]
The other, generalizing ii) above, is that the map
\[ j:  L \otimes_K H \to \End_K(L),\]
defined by $j(x \otimes h)(y) = xh(y)$, is a $K$-module isomorphism.  

As [GP87] showed, given a Galois extension $L/K$ with Galois group $\Gamma$, there is a bijection between Hopf Galois structures on $L/K$ and regular subgroups of $\Perm(\Gamma)$ normalized by the image $\lambda(\Gamma)$ in $\Perm(\Gamma)$ of the left regular representation map $\lambda:  G \to \Perm(G)$, as follows.  

Suppose $L/K$ is a Galois extension of fields with Galois group $\Gamma$, and suppose $L/K$ is also an $H$-Hopf Galois extension where $H$ is a cocommutative $K$-Hopf algebra.  Then $L \otimes_K L$ is an $L \otimes_K H$-Hopf Galois extension of $L$, hence $\Hom_L(K\Gamma, L)$ is a $L \otimes_K H$-Hopf Galois extension of $L$.  Then $L \otimes_K H$ is isomorphic to a group ring $LN$ where the group $N$ acts on $\Hom_L(K\Gamma, L)$ by acting on the group ring $K\Gamma$ as a regular group of permutations of $\Gamma$.   The group $N$ has the property that it is normalized by the image $\lb(\Gamma)$ in $\Perm(\Gamma)$ of the left regular representation map $\lambda:  \Gamma \to \Perm(\Gamma)$.

Conversely, if $N$ is a regular group of permutations of $\Gamma$, then $L \otimes_K L$ is a $LN$-Hopf Galois extension of $L$ where  $N$ acts by acting on $K\Gamma$.  If in addition, $N$ is normalized by $\lb(\Gamma)$ in $\Perm(\Gamma)$, then by Galois descent, there is a $K$-Hopf algebra $H$ so that $L \otimes_K H \cong LN$ and the Hopf Galois structure of $LN$ on $(L \otimes_K L)/L$  arises from an $H$-Hopf Galois structure on $L/K$.  

If the group $N$ is isomorphic to a given abstract group $G$, we say that $H$ has type $G$.

Given an $H$-Hopf Galois structure on $L/K$, there is a Galois correspondence from $K$-subHopf algebras of $H$ to intermediate fields $E$ with $K \subseteq E \subseteq L$, given by $H_0 \mapsto L^{H_0}$ the set of elements of $L$ fixed under the action of $H_0$ on $L$.    For the classical Galois structure on $L/K$ given by the the $K$-Hopf algebra $H = K\Gamma$, the Galois correspondence is bijective, and so the number of intermediate fields is equal to the number of subgroups of $\Gamma$.  But  for a general $H$-Hopf Galois structure by a $K$-Hopf algebra $H$ on $L/K$, it is known by [CS69] only that the Galois correspondence map is injective.  The main point of this paper is to determine how far from surjective the Galois correspondence is.  If we denote $\mathfrak{F}_\Gamma$, resp. $\mathfrak{F}_H$ the image of the Galois correspondence for the Hopf Galois structure given by $K\Gamma$, resp. $H$, then we are interested in $|\mathfrak{F}_H|$ and $|\mathfrak{F}_\Gamma|$:  their ratio measures the proportion of intermediate fields of $L/K$ that are in the image of $\mathfrak{F}_H$.

Suppose $L \otimes_K H = LN$ for $N$ a regular subgroup of $\Perm(\Gamma)$ normalized by $\lambda(\Gamma)$.  By [CRV16], Theorem 2.3, the set of $K$-subHopf algebras of $H$ is bijective with the set of subgroups $M$ of $N$ that are $\lambda(\Gamma)$ invariant.    So $|\mathfrak{F}_H|$ is equal to the number of $\lambda(\Gamma)$-invariant subgroups of $N$.  

To determine the latter, we translate the problem to studying subgroups of the holomorph of $G$.

Suppose $L/K$, with Galois group $(\Gamma, \cdot)$, has an $H$-Hopf Galois structure, and  $L \otimes_K H \cong LN$ as above.  Suppose $H$ is of type $(G, \star)$:  that is, there exists an isomorphism of groups $\alpha:  (G, \star) \to N \subset \Perm(\Gamma)$.    Since $N$ is a regular subgroup of $\Perm(\Gamma)$, the map $a:  G \to \Gamma$ given by 
\[ a(g) = \alpha(g)(e) \]
is a bijection, where $e$ is the identity element of $\Gamma$.  Moreover, $\alpha$ may be recovered from $a$:  
\begin{lemma}\label{3.1} For all $g$ in $G$, 
 \[ \alpha(g) = a \lambda_\star(g) a^{-1}.\] \end{lemma}
For given $g$ in $G$, $\gamma$ in $\Gamma$, we have
\beal \alpha(g)(\gamma) &= \alpha(g)a(a^{-1}(\gamma) \\
&= \alpha(g)\alpha(a^{-1}(\gamma))(e)\\
&= \alpha(g \star a^{-1}(\gamma))(e)\\
&= a \lambda_{\star}(g)a^{-1}(\gamma) .\eeal

Now we show that to the $H$-Hopf Galois structure of type $(G, \star)$ on $L/K$ with Galois group $\Gamma$, and the bijection $a$, there corresponds a left skew brace structure $(G, \star, \circ)$ on $G$.  

Let $\beta:  \Gamma \to \Perm(G)$ by 
\[ \beta(\gamma)(g) = a^{-1}\lambda_{\Gamma}(\gamma)a(g).\]
Then $\beta$ is a regular embedding of $\Gamma$ in $\Perm(G)$ since $\Gamma$ and $G$ have the same cardinality, $a$ is a bijection, and $\lambda_{\Gamma}$ is a regular embedding.  Moreover, since $\alpha(G) = N \subset \Perm(\Gamma)$ is normalized by $\lambda(\Gamma)$, we have: for all $g$ in $G$, $\gamma$ in $\Gamma$, there is an $h$ in $G$ so that
\[ \lambda_{\Gamma}(\gamma)\alpha(g) = \alpha(h)\lambda_{\Gamma}(\gamma).\]
Conjugating every term by the bijection $a^{-1}$ yields 
\[ \beta(\gamma) \lambda_{\star}(g) = \lambda_{\star}(h) \beta(\gamma).\]
Thus $\beta(\gamma)$ is in $\Hol(G)$, the normalizer in $\Perm(G)$ of $\lambda_{\star}(G)$.  

Define an operation $\circ$ on $G$, induced from the operation on $\Gamma$ via the bijection $a: G  \to \Gamma$:  
\[ g \circ h= a^{-1}(a(g) \cdot a(h)).\]
for $g, h$ in $G$.  Then\
\[ a(g \circ h) = a(g) \cdot a(h),\]
so $a:  (G, \circ) \to (\Gamma, \cdot)$ is an isomorphism.

Let $b = a^{-1}:   (\Gamma, \cdot) \to (G, \circ)$.

\begin{lemma}\label{3.2}  For $\gamma$ in $\Gamma$, $x$ in $G$,\[ \beta(\gamma)(x) = b(\gamma) \circ x.\]
\end{lemma}

For let $\gamma = a(g)$ in $\Gamma$.  Then 
\beal  \beta(\gamma)(x) &= a^{-1}(\lambda(\gamma) a(x)) \\
&= a^{-1}(\gamma \cdot  a(x)) \\
&= a^{-1}(a(g) \cdot  a(x)) \\
&= g \circ x = b(\gamma) \circ x .\eeal
Thus 
$\beta(\gamma) = \lambda_{\circ}(b(\gamma)) .$

\begin{proposition}  The group $(G, \star)$ with the additional group structure $(G, \circ)$ is a skew left brace with additive group $(G, \star)$.
\end{proposition}

\begin{proof}  Since $\beta:  \Gamma \to \Hol(G, \star)$ is a regular embedding and $\Hol(G) \cong \lambda_{\star}(G)\rtimes \Aut(G, \star)$, write $\beta(\gamma) = \beta_l(\gamma) \beta_r(\gamma)$ where $\beta_l(\gamma)$ is in $\lambda_{\star}(G)$ and $\beta_r(\gamma)$ is in $\Aut(G)$.  Then 
\[ b(\gamma) = \beta(\gamma)(e) = \beta_l(\gamma)\beta_r(\gamma)(e). \]
Since $\beta_r(\gamma)$ is an automorphism of $(G, \star)$ and $e$ is the identity element, this gives
\[ \beta_l(\gamma) = \lambda_{\star} (b(\gamma)) .\]
So $(\lambda_{\star} (b(\gamma)))^{-1}\beta(\gamma) = \beta_r(\gamma)$ is in $\Aut(G, \star)$.  
Letting $b(\gamma) = g$, then for all  $g, x, y$ in $G$, we have 
\[ \lambda_{\star} (g)^{-1}\beta(\gamma) (x\star y) = \lambda_{\star} (g)^{-1}\beta(\gamma) (x) \star\lambda_{\star} (g)^{-1}\beta(\gamma) (y).\]
Since $\beta(\gamma)(x) = b(\gamma) \circ x = g \circ x$,  this becomes
\[ g^{-1}\star (g \circ (x \star y)) = (g^{-1}\star (g \circ x)) \star (g^{-1}\star (g \circ  y))\]
for all $g, x, y$ in $G$.  This immediately yields that $(G, \star, \circ)$ is a skew left brace.  \end{proof}

Conversely, to get from a skew left brace to a Hopf Galois extension on $L/K$, we  suppose $(G, \star, \circ)$ is a skew left brace,  $L/K$ is a Galois extension with Galois group $(\Gamma, \cdot)$, and $a: (G, \circ) \to (\Gamma, \cdot)$ is an isomorphism of groups with inverse map $b:  (\Gamma, \cdot) \to (G, \circ)$.  Let 
\[ \beta:  \Gamma \to \Perm(G)\]
by  
\[ \beta(\gamma)(x) = b(\gamma) \circ x = \lambda_{\circ}(b(\gamma))(x).\]
Then $\lambda_{\star}(b(\gamma))^{-1}\lambda_\circ (b(\gamma))$ is in $\Aut(G, \star)$.  For
\beal   \lambda_{\star}(b(\gamma))^{-1}\lambda_\circ (b(\gamma)(x\star y)) 
&= b(\gamma)^{-1} \star (b(\gamma) \circ (x \star y))\\
&= b(\gamma)^{-1} \star (b(\gamma) \circ x)  \star b(\gamma)^{-1} \star (b(\gamma) \circ y)\\
&= \lambda_{\star}(b(\gamma))^{-1}\lambda_\circ (b(\gamma)(x) \star \lambda_{\star}(b(\gamma))^{-1}\lambda_\circ (b(\gamma))( y) 
\eeal
by the left skew brace property. So 
\[ (\lambda_{\star}(b(\gamma))^{-1})\lambda_{\circ }(b(\gamma) )= \theta_b(\gamma)\] 
is in $\Aut(G, \star)$, and 
\[ \beta(\gamma) = \lambda_{\star} (b(\gamma))\theta_b(\gamma)\]
 is in $\Hol(G, \star). $  It follows that for all $g$ in $G$, $\gamma$ in $\Gamma$, there is an $h$ in $G$ so that
\[ \beta (\gamma) \lambda_{\star}(g) = \lambda_{\star}(h)\beta(\gamma)\]
in $\Perm(G)$.  Defining $\alpha:  G \to \Perm(\Gamma)$ by $\alpha(g) = a \lambda_{\star}(g)a^{-1}$ as in Lemma \ref{3.1}, we have that  for all $g, \gamma$ there is an $h$ so that
\[ \lambda_\cdot (\gamma)\alpha(g) = \alpha(h)\lambda_\cdot (\gamma)\]
in $\Perm(\Gamma)$.  
Hence the image $\alpha(G)$ in $\Perm(\Gamma)$ is normalized by $\lambda_{\cdot}(\Gamma)$, hence by Galois descent [GP87] yields a Hopf Galois structure on $L/K$.

\section{$\circ$-stable  subgroups of $(G, \star)$}

Given a  Galois extension $L/K$ with Galois group $\Gamma$, a skew left brace $(G, \star, \circ)$ and an isomorphism $a:  (G, \circ) \to \Gamma$, there is an $H$-Hopf Galois structure on $L/K$ of type $(G, \star)$.  To study the image $\mathfrak{F}_H$ of the Galois correspondence for $H$, and the ratio $|\mathfrak{F}_H| / |\mathfrak{F}_{\Gamma}|$, we introduce some subgroups of $(G, \star)$.  

\begin{definition}  
A subgroup $(G', \star)$ of a skew left brace $(G, \star, \circ)$ is $\circ$-stable  (``circle-stable") if  $\lb_\star(G')$ is closed under conjugation in $Perm(G)$  by $\lb_{\circ}(g)$ for all $g$ in $G$. 
\end{definition}

Thus a subgroup $(G', \star)$ is $\circ$-stable if for all $g'$ in $G'$ and $g, x$ in $G$, there exists $h'$ in $G'$ so that
\[ \lb_{\circ}(g)\lb_\star(g')(x) = \lb_\star(h')\lb_{\circ}(g)(x),\]
or
\[ g \circ(g'\star x) = h'\star (g \circ x).\]
By the left brace property, this condition becomes
\[ (g \circ g')\star g^{-1}\star (g \circ x) = h'\star (g \circ x), \]
so $\circ$-stability of $G'$  is equivalent to:  for all $g$ in $G$, $g'$ in $G'$, the element
\[ (g \circ g')\star g^{-1} = h' \]
is in $G'$.   

We then have

\begin{proposition}\label{4.1}  A $\circ$-stable subgroup  $(G', \star)$ is also a subgroup of $(G, \circ)$, and so $(G', \star, \circ)$ is a sub-skew left brace of $G$. \end{proposition}

For suppose that for all  $g$ in $G$, $g'$ in $G'$, the element $ (g \circ g')\star g^{-1} = h'$ is in $G'$.  Since $(G', \star)$ is a subgroup of $(G, \star)$, then for all $g'', g'$ in $G'$, $g'' \circ g' = h'\star g''$ is in $G'$, so   $G'$ is closed under the operation $\circ$.  

\begin{remark}\label{4.2} The concept of a $\circ$-stable subgroup $G$ of $G'$ is similar to two other conditions in the literature.

Bachiller [Bac16a] defines a left ideal of a skew left brace $G = (G, \star, \circ)$ to be a subgroup $(G', \star)$ of $(G, \star)$ that is closed under the action of $\Lc_g$ for all $g$ in $G$.  That is,  for all $g$ in $G$, $g'$ in $G'$, 
\[ g^{-1}\star (g \circ g')  \text{  is in } G'.\]

Just as in Proposition \ref{4.1}, a left ideal is closed under the circle operation (and under taking the inverse of an element in $(G, \star)$ since all groups are finite), so is a sub-skew left brace.   

If $(G, \star, \circ)$ is a left brace, that is, if the additive group $(G, \star)$ is abelian, then a subgroup $(G, \star)$ is $\circ$-stable if and only if it is a left ideal.  

Guarneri and Vendramin [GV17, Corollary 3.3]  consider a subset  $G'$ of a skew left brace $(G, \star, \circ)$ and show that if 
\[ h\star \Lc_g(g')\star h^{-1}\text{  is in } G'\]
 for all $g, h$ in $G$, $g'$ in $G'$, then the solution of the Yang-Baxter equation on $G \times G$ associated to the skew left brace $G$ restricts to a solution on $G' \times G'$.   That condition is equivalent to 
 \[ h\star g^{-1}\star (g \circ g')\star h^{-1}  \in G'.\]
 
 If $(G', \star)$ is a subgroup (not just a subset) of the additive group $(G, \star)$, then setting $h = e$, the identity of $(G, \star)$, the [GV17] condition implies that $G'$ is a left ideal, while setting $h = g$ yields that $G'$ is $\circ$-stable.  However, a subset $G'$ of $G$ satisfying the [GV17] condition need not be a subgroup of $(G, \star)$:  see Remark \ref{final}, below. 

A ring without identity $(G, +, \cdot)$ is a radical algebra if under the circle operation  defined by 
\[ g \circ h = g + h + g\cdot h ,\]
$(G, \circ)$ is a group.  In that case, $(G, +, \circ)$ is a left brace, as was originally noted by Rump in [Rum07].
If a left brace is a radical algebra, then a subgroup $(G', +)$ of $G$ is a left ideal of the left brace $(G, +, \circ)$ if and only if it is a left ideal of the algebra.
\end{remark}

Here is our main result.

\begin{theorem}\label{main}  Let $(G, \star, \circ)$ be a skew left brace.   Let $a:  (G, \circ) \to (\Gamma, \cdot)$ be an isomorphism of groups with inverse $b: (\Gamma, \cdot) \to (G, \circ)$.  Let $L/K$ be a Galois extension with Galois group $(\Gamma, \cdot)$.  Then for the unique $H$-Hopf Galois structure on $L/K$ of type $(G,\star)$  corresponding to the isomorphism $a$,  there is a bijection  between the $K$-subHopf algebras of $H$ and the $\circ$-stable subgroups $G'$ of $(G, \star )$.
\end{theorem}

\begin{proof}
Let $L/K$ be a Galois extension of fields with Galois group $\Gamma$.   Let $(G, \star, \circ)$ be a skew left brace and  $a:  (G, \circ) \to \Gamma$  an isomorphism of groups. 

Recall that $\alpha:  G \to \Perm(\Gamma)$ is defined by
\[ \alpha(g)(\gamma) = a \lb_\star(g)a^{-1}(\gamma);\]
and $\beta:  \Gamma \to \Perm(G)$ is defined by 
\[ \beta(\gamma)(g) = b\lambda_{\cdot}(\gamma) b^{-1}(g) = b(\gamma) \circ g .\]

Suppose $(G', \star, \circ)$ is a $\circ$-stable subgroup of $(G, \star)$.  This means:  for all $g, x$ in $G$, $g'$ in $G'$, there exists $h'$ in $G'$ so that 
 \[ g \circ (g'\star x) = h'\star (g \circ x). \qquad \qquad (*) \]
We show that $\alpha(G')$ is normalized by $\lambda(\Gamma)$, and conversely, if $\alpha(G')$ is a subgroup of $\alpha(G)$ normalized by $\lambda(\Gamma)$, then $G'$ is a $\circ$-stable subgroup of $G$.  
  
Starting from the  $\circ$-stable  equation $(*)$, let  $a(g) = \gamma$.  Then for $x$ in $G$, 
\[ b(\gamma) \circ \lb_\star(g')(x) = \lb_\star(h')(b(\gamma) \circ x )\]
or, since $b(\gamma) \circ y = \beta(\gamma)(y)$ for all $y$ in $G$, 
\[ \beta(\gamma)\lb_\star(g') = \lb_\star(h')\beta(\gamma).  \qquad \qquad (**) \]

Conjugating each map in the normalizing equation (**)  by $a$ gives the equation
\[ \lb_\cdot(\gamma) \alpha(g') \lb_\cdot(\gamma)^{-1} = \alpha(h') \qquad \qquad (***)\]
Thus the condition that $(G', \star, \circ)$ is a $\circ$-stable subgroup of $(G, \star, \circ)$ holds if and only if for all $g, x$ in $G$, $g'$ in $G'$, there exists $h'$ in $G'$ so that equation $(***)$ holds, that is, if and only if 
 $\alpha(G')$ is a $\lb_\cdot(\Gamma)$-stable subgroup of $\alpha(G)$.  

From [CRV16, Theorem 2.3], the $\lb_\cdot(\Gamma)$-stable subgroups of $\alpha(G)$ correspond  to the $K$-subHopf algebras of the $K$-Hopf algebra obtained by Galois descent from the $L$-Hopf algebra $L[\alpha(G)]$.  \end{proof}

Thus for the $H$-Hopf Galois structure on the $\Gamma$-Galois extension $L/K$ corresponding to the skew left brace $(G, \star, \circ)$, 

$|\mathfrak{F}_\Gamma |$ = the number of subgroups of $\Gamma \cong (G, \circ)$, while

$|\mathfrak{F}_H |$ = the number of $\circ$-stable subgroups of $(G, \star)$, which by Proposition \ref{4.1} are subgroups of $(G, \circ)$.  

In the remainder of the paper we consider several examples of skew left braces $(G, \star, \circ)$ and determine the number of subgroups of $(G, \circ)$  and the number of  $\circ$-stable subgroups of $(G, \star)$.   Two examples arise from  non-commutative radical $\Fp$-algebras of dimension 3 as described in [DeG17], one is the left brace of order 8 in [Rum07], and several arise from skew left braces arising from the exact factorization of a group into a product of two subgroups, a special case of Hopf Galois structures constructed by a pair of fixed point free homomorphisms.

\section{Examples from non-commutative nilpotent algebras}

Let $p$ be odd and consider each of the two isomorphism types of non-commutative nilpotent $\Fp$-algebras of dimension 3, as described by Section 5 of [DeG17].  

\begin{example}  Let $A$ be the radical algebra $A_{3, 5} = \langle x, y \rangle$, generated as an $\Fp$-module by elements $x, y, z,$ where $xy = z, yx = -z$ and all other products among the basis elements are zero.  Then $A$ is a group under the operation $\circ$, defined by
$ u \circ v = u + v + uv.$
Then $z \circ x = z + x = x + z = x \circ z$ and $y \circ z = y + z = z + y = z \circ y$, so $z$ is in the center of the group $(A, \circ)$.  It is easily checked that $\langle z \rangle$ is the center of $(A, \circ)$ and that the quotient is isomorphic to $C_p \times C_p$.  

\begin{lemma}  For $A = A_{3, 5}$, $(A, \circ)$ is isomorphic to the Heisenberg group  $Heis_3(\Fp)$ of $3 \times 3$ upper triangular matrices in $\M_3(\Fp)$ with diagonal entries all equal 1. \end{lemma}

One way to see this is to show that every element of $(A, \circ)$ has order dividing $p$, and cite [Con] that the only non-abelian group of order $p^3$ with every element of order dividing $p$ is the Heisenberg group.

To see that every element of $(A, \circ)$ has order dividing $p$, we observe that
\beal &(ax + by + cz) \circ (a'x + b'y + c'z) \\&= (a + a')x + (b + b')y + (c + c')z + (ax+ by +cz)(a'x + b'y + c'z)\\
&= (a + a')x + (b + b')y + (c + c')z + (ab' - ba')z.\eeal
So identifying $ ax + by + cz$ as the  vector $(a, b, c)^T = \bpm a\\b\\c \epm$, we have
\beal\bpm a\\b\\c \epm \circ \bpm ra\\rb\\rc \epm &= \bpm a\\b\\c \epm + \bpm ra\\rb\\rc \epm + \bpm 0\\0\\rab - rba\epm\\&=  
\bpm (r+1)a\\(r+1)b\\(r+1)c \epm \eeal
Hence $ ax + by + cz$ has order dividing $p$ for all $a, b, c$ in $\Fp$.  

We have
\begin{proposition}  The Heisenberg group $Heis_3(\Fp)$ has $2p^2 + 2p +4$ subgroups. \end{proposition}  

\begin{proof}  There is of course one subgroup of order 1. Each of the $p^3 -1$ non-zero elements of $Heis_3(F_p)$ generates a subgroup of order $p$, and each such subgroup has $p-1$ generators.  Thus 
there are $p^2 + p + 1$ subgroups of order $p$.  

If  $H$ is a subgroup of order $p^2$ and $ax + by + cz$ and  $a'x + b'y + c'z$ is a minimal generating set for $H$, then we can do the same kind of manipulations to those generators as we could if they were elements of $\Fp^3$.  So we can assume that  the subgroup $H$ has one of the forms
\[\left \langle \bpm 1\\0\\c\epm, \bpm 0\\ 1\\ c'\epm \right \rangle, 
\left \langle \bpm 1\\d\\0\epm, \bpm 0\\ 0\\ 1\epm  \right\rangle, 
\left \langle \bpm 0\\1\\0\epm, \bpm 0\\ 0\\ 1\epm  \right\rangle, \]
 for $c, c', d$ in $\Fp$.
  So there are  $p^2 + p  + 1$ subgroups of order $p^2$.   Since there is one subgroup of order $p^3$, we have the claimed number of subgroups.  \end{proof}

By Remark \ref{4.2}, the $\circ$-stable subgroups of $(A, +)$ are the left ideals of $A$.  

\begin{proposition}  The algebra $A = A_{3, 5}$ has  $p + 4$ left ideals. \end{proposition}

\begin{proof}  Every left ideal is an additive subgroup of $A$, so we can assume that an ideal is generated as a subgroup as described in the last proof.
  
The zero subgroup is an ideal.

If $J$ is a left ideal and contains $\alpha = ax + by + cz$, then $J$ also contains $x\alpha = bz$ and $y\alpha = -az$, hence $J$ contains $(0, 0, 1)^T$.   So the cyclic subgroups generated by $(1, b, c')^T$ or $(0, 1, c')^T$ are not left ideals, but the cyclic subgroup generated by $z = (0, 0, 1)^T$ is an ideal.
 Thus the only group of order $p$ that is an ideal is $J = \langle z \rangle$.

Suppose $J$ is a left ideal and contains a subgroup of order $p^2$.  If $H$ has generators $x + cz = (1, 0, c)^T$  and $y + c'z = (0, 1, c')^T$, then $J$ also contains  $xy = z = (0, 0, 1)^T$, so $J = A$.  If $H$ has generators  $x + by  = (1, b, 0)^T$ and $z = (0, 0, 1)^T$, or generators  $y = (0,1,0)^T$ and $z = (0, 0, 1)^T$, then $H$ is closed under left multiplication by $x$ and $y$ (and, of course, $z$), so is an ideal.
Thus $A$ has $p + 1$ ideals of order $p^2$.  

Since $A$ is an ideal, $A$ has $p + 4$ left ideals, as claimed.
\end{proof}

To summarize, there are $2p^2 + 2p + 4$ subgroups of $(A,  \circ) \cong \Gamma$ and $p + 4$ $\circ$-stable subgroups of the skew left brace $(A, +, \circ)$.  

\end{example}

\begin{example}
Now we consider the algebras  $A = A^{\delta}_{3, 4}$ of De Graaf, generated by $x, y, z$ as an $\Fp$-module with $x^2 = z, y^2 = \delta z,  xy = z, yx = 0$  and $zw = wz = 0$ for all $w$ in $A$.  Here $\delta$ is an arbitrary element of $\Fp$.  Each choice of $\delta$ gives a different isomorphism type of algebras.  But the results come out the same.

\begin{proposition}  Let $(A, \circ)$ be the circle group on $A = A^{\delta}_{3, 4}$.  Then $(A, \circ) \cong Heis_3(\Fp)$. \end{proposition}

\begin{proof}  It suffices to show that every element of $(A, \circ)$ has order $p$.  That follows from the relation
\[ \bpm a\\b\\c \epm^{\circ r} = \bpm ra\\rb\\rc + \binom r2 (a^2 + ab + \delta b^2)\epm,\]
which is easily verified for all $r \ge 1$ by induction.
\end{proof}

The subgroups of $(A, \circ)$ are the same as in the last example.   So there are $ 2p^2 + 2p + 4$ subgroups of $(A, \circ)$.  

 Of those subgroups, 
\[ (0),\left \langle \bpm 0\\0\\1 \epm \right\rangle,\left \langle \bpm 1\\d\\0\epm, \bpm 0\\ 0\\ 1\epm \right\rangle, \left \langle \bpm 0\\1\\0\epm, \bpm 0\\ 0\\ 1\epm \right\rangle, \text{  and  } A \]
are ideals.  So we have, just as with $A_{3, 5}$,
\begin{proposition}  For $A = A^{\delta}_{3, 4}$, there are $2p^2 + 2p + 4$ subgroups of $(A, \circ)$ and $p + 4$ left ideals of the radical algebra $A$.  \end{proposition}

Translating these results into a statement about Hopf Galois extensions, we have:

\begin{corollary}\label{5.8}  Let $L/K$ be a Galois extension of fields with Galois group $\Gamma \cong Heis_3(\Fp)$.  Suppose $L/K$ has an $H$-Hopf Galois structure of elementary abelian type that corresponds to one of the $p+1$ isomorphism types of dimension 3 non-commutative nilpotent associative $\Fp$-algebras.  Then the Galois correspondence from $K$-sub-Hopf algebras of $H$ to intermediate fields between $K$ and $L$  maps onto exactly $p+4$ of the $2p^2 + 2p + 4$ intermediate subfields. \end{corollary} 

Since the Galois correspondence from $K$-subHopf algebras of $H$ to intermediate fields is injective, this result follows immediately from the counts above and Theorem \ref{main}.

\section{Rump's brace of order 8}

Example 2 of [Rum07] is a left brace $A = (A, +, \circ)$ with additive group $(A, +)$ isomorphic to $(\FF_2^3, +) \cong C_2^3$ and circle group $(A, \circ)$  isomorphic to the dihedral group $D_4$.   Thus if $L/K$ is a Galois extension with Galois group $\Gamma$ isomorphic to $D_4$, then corresponding to the brace $A$ is an $H$-Hopf Galois structure on $L/K$ where the $K$-Hopf algebra $H$ is of elementary abelian type:  that is, $L \otimes_K H = LN$ where $N \cong  (\FF_2^3. +)$.  The brace $A$ is not a ring, as Rump notes.

\begin{proposition}  The group $(A, \circ) $ has 10 subgroups, of which three  are $\circ$-stable subgroups of the brace $(A,+, \circ)$.   \end{proposition}

\begin{proof} 

To find the $\circ$-stable subgroups of $A$, we'll identify $(A, \circ)$ with $(D_4, \cdot)$ using the identification of elements in $\FF_2^3$ with elements of $D_4$ according to the following table from  Example 2 of [Rum07].     Let $D_4 = \langle c, s \rangle$ where $c^4 = s^2 = e$, the identity, and $cs = sc^3$. Elements of $\FF_2^3$ are written as  $abc$ with $a, b, c $ in $\FF_2 = \{ 0, 1\}$. 
\begin{center} \begin{tabular}
{c | c}\hline
$e$ & 000\\
$c$ & 011\\
$c^2$ & 001\\
$c^3$ & 101\\
$s$ & 100\\
$sc$ & 110\\
$sc^2$ & 010\\
$sc^3$ & 111\\
\end{tabular} \end{center}

Then the addition in $D_4$ induced from that on $\FF_2^3$ makes $(D_4, +, \cdot)$ into a left brace.  Here is the addition table for $(D_4, +)$:  

\begin{center} \begin{tabular}
{c | c|c|c|c|c|c|c|c}
sum 		&$e$& $c $				&$c^2 $&$c^3 $	&$s $&$sc $			&$sc^2 $&$ sc^3$\\\hline
$e$ 		&$e$& $c $				&$c^2 $&$c^3 $	&$s $&$sc $			&$sc^2 $&$ sc^3$\\\hline
$c$ 		&$c $ &$ e$			&$sc^2 $&$sc $	&$sc^3 $&$c^3 $	&$c^2 $&$ s$\\\hline
$c^2$ 	&$ c^2$ &$ sc^2$&$e $&$s $				&	$c^3 $&$sc^3 $&$c $&$sc $\\\hline
$c^3$	& $c^3 $ &$sc $		&$s $&$e $			&$c^2 $&$ c$		&$sc^3 $&$sc^2 $\\\hline
$s$ 		&	$s $ &$sc^3 $		&$c^3 $&$c^2 $	&$e $&$sc^2 $		&$sc $&$c $\\\hline
$sc$ 		& $sc $ &$c^3 $	&$sc^3 $&$ c$		&$sc^2 $&$e $		&$s $&$c^2 $\\\hline
$sc^2$ & $sc^2$ &$c^2 $	&$c $&$sc^3 $		&$sc $&$s $			&$e $&$c^3 $\\\hline
$sc^3$ &$sc^3 $ &$s $		&$sc $&$sc^2 $	&$c $&$c^2 $		&$c^3 $&$e $\\\hline
\end{tabular}\end{center}

A left ideal of $(D_4,  +,\cdot)$ is a subgroup of $(D_4, +)$  with the property that  $-g + g \cdot x$ is in $L$ for all $g$ in $A$, $x$ in $L$.  Since the additive group of $A$ is that of a vector space over $\FF_2$, $-g + g \cdot x = g + g \cdot x$.

To see what this looks like, we let $g = c$ and find $c + c \cdot x$ for all $x$ in $A$:
\begin{center} \begin{tabular}
{c | c}
$x$ & $c  + c \cdot x$ \\ \hline
$e$ &$e$\\
$c$ & $sc^2$\\
$c^2$ &$sc $\\
$c^3$ & $c $\\
$s$ &$s $\\
$sc$ & $sc^3 $\\
$sc^2$ &$ c^3$\\
$sc^3$ & $c^2 $\\
\end{tabular} \end{center}

By Proposition 4.1, the left ideals are among the ten subgroups of $(D_4, \cdot)$, namely:
\beal & \langle e \rangle, \langle c^2\rangle, \langle s\rangle, \langle sc \rangle, \langle sc^2\rangle,\langle sc^3 \rangle, \\&\{e, c, c^2, c^3 \},\{e, s, c^2, sc^2 \},\{c, sc, c^2, sc^3 \},D_4.\eeal
From the last table it is immediate that no subgroup of order 2 of $(D_4, \cdot)$ is a left ideal except possibly $\langle s \rangle$.  But $c^2 + c^2 \cdot s = c$.  So $\langle s \rangle$ is not a left ideal.  This last table also rejects 
$\{ e, c, c^2, c^3\}$ and $ \{ e, s, c^2, sc^2\}$ as left ideals.  

Thus, other than the trivial subgroups $A$ and $\{e\}$, the only potential left ideal of the brace $(D_4, +, \cdot)$ is the subgroup $\{ e, sc, c^2,sc^3\}$.  That subgroup is a subgroup of $(D_4, +_) $ since it corresponds to the subgroup $\{ 000, 110, 001, 111 \}$ of $(\FF_2^3, +) \cong (D_4, +)$, and a routine check shows that it is in fact a left ideal of the left brace $(D_4, +, \cdot)$.  That completes the proof.
\end{proof}
\end{example}

\section{Skew left braces arising from fixed point free pairs of homomorphisms}

Let $L/K$ be a Galois extension with Galois group $\Gamma$, and let $G$ be a group with $|G| = |\Gm|$.  In this section we consider Hopf Galois extensions on $L/K$ of type $G$ arising from fixed point free pairs of homomorphisms. 

\begin{definition}  Given two groups $\Gamma$ and $G$ of the same finite cardinality, and homomorphisms 
\[ f_l, f_r:  \Gamma \to G,\]
$(f_l, f_r)$ is a fixed point free pair of homomorphisms if 
\[f_l(\gamma) = f_r(\gamma) \iff \gamma = e_\Gamma, \]
where $e_\Gamma$ is the identity element of $\Gamma$.  \end{definition}
Since $\Gamma$ and $G$ have the same cardinality, $(f_l, f_r)$ is a fixed point free pair if and only if the set $\{f_l(\gamma)f_r(\gamma)^{-1} :  \gamma \text{ in } \Gamma \} = G$.    

If $(f_l, f_r)$ is a fixed point free pair of homomorphisms from $\Gamma$ to $G$, then 
\[ \beta:  \Gamma \to \Hol(G),\]
given by $\beta(\gamma) = \lambda(f_l(\gamma)) \rho(f_r(\gamma))$ is a regular embedding, hence yields a Hopf Galois structure of type $G$ on a Galois extension $L/K$ with Galois group $\Gamma$.  In fact, for $x$ in $G$, $f_l(\gamma) = g_l, f_r(\gamma) = g_r$, 
\beal  \beta(\gamma)(x) &= \lambda(g_l) \rho(g_r)(x) 
\\&= g_lxg_r^{-1} = g_lg_r^{-1}g_rxg_r^{-1} 
\\& = \lambda(g_lg_r^{-1})C(g_r) \eeal
in $\InHol(G) = \lambda(G)\rtimes\Inn(G)$, 
where $C(g)$ is conjugation by $g$ and $\Inn(G)$ is the group of inner automorphisms of $G$.

These Hopf Galois structures were studied for $\Gamma = G$ in [CCo07] , generalizing ideas from [CaC99], and in general in Section 2 of [BC12].  

A natural class of examples, arising in [By15], is the following:

Let $G$ be a finite group with two subgroups $H$ and $J$ so that $|H||J| = |G|$ and $H\cap J = \{e\}$.  The subgroups $H$ and $J$ are called complementary  in $G$ in Section 7 of [By15].  

Suppose $G$ has complementary subgroups $H$ and $J$, and let $\Gamma = H \times J = \{ (g_l, g_r)\}$  with $g_l$ in $H$, $g_r$ in $J$.  Define $f_l, f_r:  \Gamma \to G$ by 
\beal  f_l(g_l, g_r) &= g_l, \\
 f_r(g_l, g_r) &= g_r .\eeal
 Then $(f_l, f_r)$ is a fixed point free pair.  For $\gamma = (g_l, g_r)$ define 
 \[ \beta: \Gamma\to \Perm(G)\]
 by 
\[\beta(\gamma)  = \lambda(f_l(\gamma))\rho(f_r(\gamma) )= \lambda(g_l)\rho(g_r)\]
where $\rho$ is the right regular representation of $G$ in $\Perm(G)$.  Thus for $x$ in $G$ and $\gamma = (g_l, g_r)$ in $\Gamma$, 
\[\beta(\gamma)(x) = \beta(g_l, g_r)(x)  = g_lxg_r^{-1}.\]

Corresponding to $\beta:  \Gamma \to \Hol(G)$ is the map $b: \Gamma \to G$ defined by 
\[ b(g_l, g_r) = \beta(g_l, g_r)(e) = g_lg_r^{-1},\]
with inverse $a(g) = a(g_lg_r) = (g_l, g_r^{-1})$.  Then for $g = g_lg_r$ and $h = h_lh_r$  in $G$, 
\beal g \circ h &=  b(a(g) a(h)) \\
&=b((g_l, g_r^{-1}),(h_l, h_r^{-1})) \\
& = b ((g_lh_l, g_r^{-1}h_r^{-1})) = (g_lh_l)(g_r^{-1}h_r^{-1})^{-1}\\
&= g_l h_lh_rg_r = g_l hg_r.\eeal
This defines a left skew brace structure $(G, \cdot, \circ)$ on $G = (G, \cdot)$, where $a:  (G, \circ) \to \Gamma$ is an isomorphism of groups. 

In [SV18, Section 2], the skew brace $(G, \cdot, \circ)$  just defined is called the skew brace on $G$ arising from the exact factorization of $G$ into $H$ and $J$, where the additive group  of the skew brace is $(G, \cdot)$.  

\begin{proposition} \label{7.1} Let $G$ be a group with complementary subgroups $G_l$ and $G_r$.  Let  $\Gamma = G_l \times G_r$, and define $\beta: \Gamma \to G$ by 
\[ \beta((g_l, g_r)(\gamma) = \lambda(f_l(\gamma))\rho(f_r(\gamma))\]
in $\Hol(G)$.  Then the  $\circ$-stable subgroups of $G$ are the subgroups of $G$ that are closed under conjugation by elements of $G_l$. \end{proposition}

\begin{proof}  
Let $G'$ be a subgroup of $G = HJ$.  Then $G'$ is $\circ$-stable if for all $x$ in $G'$ and all $g = g_lg_r$ in $G$, there exists $y$ in $G'$ so that  
\[ g \circ x = yg.\]
This is true if and only if 
\[  g_l x g_r = y g_lg_r ,\]
if and only if
\[  y = g_l x g_l^{-1} = C(g_l)x.\]

\end{proof} 

Here is a class of examples.

Let $\Gamma = A \times \Delta$ and $G = A \rtimes \Delta$ where $A$ is simple and $\Delta$ is abelian.  Then a  subgroup $G'$ of $G$ is $\circ$-stable if and only if  for all $a$ in $A$, $x$ in $G'$, 
\[ axa^{-1} = y \]
is in $G'$.   If there exists a non-trivial element of $A \cap G'$, then since  $G'$ is closed under conjugation by elements of $A$  and $A$ is simple, $A \subseteq G'$.  If for some $x$ in $G'$ and $a$ in $A$, $axa^{-1} \ne x$, then since $A$ is normal in $G$, $a(xa^{-1}x^{-1})$ is a non-trivial element of $A$, hence is a non-trivial element of $A\cap G'$.  So $G'$ contains $A$.  

\begin{example}

 Let $G = Z_p \rtimes \Delta$ where $\Delta$ is a non-trivial subgroup of $Z_p^{\times}$.  Then the $\circ$-stable subgroups of $G$ are $(1)$ and the subgroups of $G$ containing $Z_p$.   

To see this, let $G = \langle a, \delta :  a^p = \delta^k = 1, \delta a = a^b \delta \rangle$  where the order of $b$ in the group $U_p$ of units modulo $p$ is $k > 1$, and let $Z_p = \langle a \rangle $ (written multiplicatively).  If $a^r\delta^s$ is in $G'$ with $\delta^s \ne 1$ (so $1 \le s < k$), then 
\[ a^{-1}a^r\delta^s a = a^{r-1}a^{b^s}\delta^s. \qquad \qquad (*)\]
If the right side is $= a^r\delta^s$, then  $a^{r-1}a^{b^s} = a^r$, so $a^{b^s} = a$, hence  $k$ divides $s$ and $\delta^s = 1$.  So if $\delta^s\ne 1$, then 
\[ (a^{-1}(a^r\delta^s)a)(a^r\delta^s)^{-1} \ne 1 \] 
and is a non-trivial element of $G' \cap Z_p$.  Thus $G'$ contains $Z_p$.   
If $G'$ is a subgroup of $G$ containing $Z_p$, then clearly $G'$ is closed under conjugation by elements of $Z_p$, so is $\circ$-stable.  
\end{example}

\begin{example} 
Let $G = S_n = A_n \rtimes Z_2$ and $\Gamma = A_n \times Z_2$, where $A_n$, $S_n$ is the alternating, resp. symmetric group and $n \ge 5$.   Then the only $\circ$-stable subgroups of $G$ are $(1), A_n$ and $S_n$.  

To see this, let $G'$ be a non-trivial $\circ$-stable subgroup of $G$.  It suffices to show that $G' \cap A_n$ is non-trivial.  Let $\tau \ne 1 \in G'$.  If $\tau$ is even, then  $G' \cap A_n$ is non-trivial.  If $\tau$ is odd and $a\tau a^{-1} \ne \tau$ for some $a$ in $A_n$, then $a\tau a^{-1}\tau^{-1} \ne 1$ and is even, so  $G' \cap A_n$ is non-trivial.  So suppose  $\tau$ is odd, and $a \tau a^{-1} = \tau$ for all $a$ in $A_n$.  Since $\tau$ is odd, every $b$ in $S_n \setminus A_n$ has the form $b = a\tau$ for some $a$ in $A_n$.  Then $b \tau b^{-1} = \tau$ for all $b$ in $S_n$, so $\tau$ is in the center of $S_n$, impossible.  So $G'$ must contain $A_n$.    

\end{example}

For a different application of Proposition \ref{7.1}, we have
\begin{example}\label{7.4}  
Let  $A = \Fp^3$.  Define a left skew brace structure $(A,  \star,\circ)$ on $A$ where  $(A, \circ) = \Fp^3$ with $\circ = +$, the usual vector space addition, and  $(A, \star)$ is isomorphic to the Heisenberg group $Heis_3(\Fp)$.  Write
\beal  Heis_3(\Fp) &= \Fp^2 \rtimes \Fp = \{ [\bpm c\\b \epm, \bpm 1 &a\\0&1 \epm] :  a, b, c \in \Fp\} \\
&\cong  \bpm 1&a&c\\0&1&b\\0&0&1\epm,\eeal
the latter viewed as a subgroup of $\GL_3(\Fp)$.  Thus the multiplication  on vectors in $\Fp^3$ is by 
\[ \bpm a\\b\\c \epm \star \bpm a'\\b'\\c' \epm =  \bpm a + a'\\b + b'\\c +c' + ab' \epm.\]
Define $f_l, f_r:  \Fp^2 \times \Fp \to Heis_3(\Fp)$ by
\beal  f_l((a,b,c)^T ) &= [\bpm c\\b\epm, I],\\
f_r((a, b, c)^T) & = [\bpm 0\\0 \epm, \bpm 1 & -a\\0 & 1\epm ].
 \eeal
Then $(f_l, f_r)$ is a fixed point free pair of homomorphisms, hence makes $A= \Fp^3$ into a skew left brace $(A, \star, \circ)$ with $(A, \star) \cong Heis_3(\Fp)$ (the additive group), and $(A, \circ) \cong (\Fp^3, +)$.  

We show
\begin{proposition}\label{7.5}  There are $2p + 4$ $\circ$-stable subgroups of $(A, \star, \circ)$. 
\end{proposition}

\begin{proof} Given a subgroup $G'$ of $G = (A, \star)$ we need to check to see if for all $h$ in $G'$, $g = g_l\star g_r$ in $G$, 
$g_l \star h\star g_l^{-1}$  is in $G'$.

\noindent For $g = \bpm a\\b\\c \epm$,  $g_l =  \bpm 0\\b\\c \epm$, so 
\beal  g_l\star h \star g_l^{-1} &= \bpm 0\\b\\c \epm \star \bpm r\\s\\t \epm \star \bpm 0\\-b\\-c \epm \\
&= \bpm r\\s\\t - rb \epm. \eeal  
So a subgroup $G'$ of $(G, \star)$ is $\circ$-stable if and only if for all $b$ in $\Fp$, 
\[ \text{if }\bpm r\\s\\t \epm \text{ is in } G', \text{ then} \bpm r\\s\\t - rb \epm 
\text{ is in } G'.\]
By Proposition \ref{4.1}, $G'$ is a subgroup of $(G, \circ) = (\Fp^3, +)$.  So if $r \ne 0$, then $G'$ must contain $\bpm 0\\0\\1 \epm$. 
 
In Section 5 we found that the subgroups of $Heis_3(\Fp) = (\Fp^3, \star)$ are of the following forms:

\beal &\langle 0\rangle,
\left \langle \bpm 1\\b\\c  \epm \right \rangle, 
\left\langle \bpm 0\\1\\c  \epm \right \rangle, 
\left\langle \bpm 0\\0\\1  \epm \right \rangle,\\& 
\left\langle \bpm  1\\0\\c \epm, \bpm 0\\1\\c'   \epm \right\rangle, 
\left\langle \bpm 1\\c\\0  \epm, \bpm 0\\0\\1   \epm \right\rangle, 
\left\langle \bpm 0\\1\\0  \epm, \bpm 0\\0\\1  \epm  \right\rangle,
Heis_3(\Fp) .\eeal 
Of these, the $2p + 4$ $\circ$-stable subgroups of $(G, \star) = Heis_3(\Fp)$ are $\langle 0\rangle$, $Heis_3(\Fp)$ and 
\[\left \langle \bpm 0\\1\\c  \epm \right\rangle,
\left\langle \bpm 0\\0\\1  \epm \right\rangle,
\left\langle \bpm 1\\c\\0  \epm, \bpm 0\\0\\1  \epm \right\rangle, 
\left\langle \bpm 0\\1\\0  \epm, \bpm 0\\0\\1  \epm\right \rangle. \]
\end{proof}
In Section 5 we took a Galois extension $L/K$ with Galois group $\Gamma \cong Heis_3(\Fp)$ and looked at Hopf Galois structures on $L/K$ of abelian type $G = (\Fp^3, +)$.   Example \ref{7.4} corresponds to a Galois extension $L/K$ with abelian Galois group $\Gamma \cong (\Fp^3, +)$ with a Hopf Galois structure of type $G = Heis_3(\Fp)$.  For Example \ref{7.4},  the number of intermediate fields $|\mathfrak{F}_{\Gamma}| = \text{ number of subgroups of  }\Fp^3 = 2p^2 + 2p + 4$, while Proposition \ref{7.5} shows that the number of fields in the image of the Galois correspondence for $H$ is $|\mathfrak{F}_H| = 2p + 4$.  For the examples in Section 5, $|\mathfrak{F}_{\Gamma}| = 2p^2 + 2p + 4$ while $|\mathfrak{F}_H| = p + 4$. 
\end{example}

\begin{remark}\label{final}  Let $G'$ be a subset of $G$ so that Guarneri and Vendramin's condition from Remark \ref{4.2} holds:

For all $g, h$ in $G$, $g'$ in $G'$,
\[ hg^{-1}(g \circ g')h^{-1} \text{  is in } G'.\]
Consider our last example $(G, \star, \circ)$ where $(G, \circ) = (\Fp^3, +)$ and $(G, \star) = Heis_3(p)$ with operation
\[  \bpm a\\b\\c \epm \star \bpm a'\\b'\\c' \epm = \bpm a+a'\\ b + b'\\ c + c' + ab' \epm. \] 
Let 
\[ h = \bpm  a'\\b'\\c' \epm, g = \bpm a\\b\\c \epm \text{ and  } g' = \bpm  r\\s\\t \epm.\]
Then 
\[ hg^{-1}(g \circ g')h^{-1} = \bpm r\\s\\t - rb' +a's - as \epm.\]
From this computation it is easy to see that for each $t$ in $\Fp$ the singleton set $\{\bpm 0\\0\\t \epm \}$ satisfies the [GV17] condition for each $t$ in $\Fp$, as does the coset
\[  \{ \bpm r\\s\\t \epm | t \text{ in } \Fp \}\]
of the subgroup $\langle \bpm 0\\0\\1 \epm \rangle$ for each fixed $r, s$ in $\Fp$.
Thus subsets of $(G, \star)$ satisfying the [GV17] condition need not be subgroups of $(G, \star)$. 
\end{remark}

\end{document}